\documentclass[aop, preprint, 12pt]{imsart}

\RequirePackage[OT1]{fontenc}
\RequirePackage{amsthm,amsmath}
\RequirePackage[numbers]{natbib}
\RequirePackage[colorlinks,citecolor=blue,urlcolor=blue]{hyperref}
\usepackage{amsfonts}
\usepackage{color}
\usepackage{amssymb}
\usepackage[english]{babel}
\usepackage[T1]{fontenc}
\usepackage{graphicx}
\usepackage{subfigure}
\usepackage{latexsym}
\usepackage{amstext}
\usepackage{mathrsfs}
\usepackage{hyperref}
\usepackage{dsfont}
\usepackage{graphics}
\usepackage{enumerate}
\usepackage{paralist}
\usepackage{color}
\usepackage{fullpage}
\newtheorem{prop}{Proposition}[section]
\newtheorem{rem}[prop]{Remark}
\newtheorem{lem}[prop]{Lemma}
\newtheorem{theo}[prop]{Theorem}

\numberwithin{equation}{section}

\newcommand{\beq}{\begin{eqnarray}}
\newcommand{\beqq}{\begin{eqnarray*}}
\newcommand{\eeq}{\end{eqnarray}}
\newcommand{\eeqq}{\end{eqnarray*}}

\startlocaldefs
\numberwithin{equation}{section}
\theoremstyle{plain}

\endlocaldefs

\begin{document}

\begin{frontmatter}
\title{Stable windings at the origin}
\begin{center}
{\it Dedicated to the memory of Marc Yor}
\end{center}
\runtitle{Stable windings at the origin}

\begin{aug}
\author{\fnms{Andreas E. Kyprianou}\thanksref{t1}\ead[label=e1]{a.kyprianou@bath.ac.uk}}
\and
\author{\fnms{Stavros M. Vakeroudis}\ead[label=e2]{stavros.vakeroudis@gmail.com}}


\thankstext{t1}{Supported by EPSRC grant EP/L002442/1}

\affiliation{University of Bath and University of the Aegean}

\address{Department of Mathematical Sciences \\
University of Bath\\
 Claverton Down\\ Bath, BA2 7AY\\
 UK.
\printead{e1}
}

\address{Department of Mathematics \\
Track: Statistics and Actuarial-Financial Mathematics \\
University of the Aegean\\
   83200 Karlovasi, Samos \\
    Greece.
\printead{e2}
}
\end{aug}

\begin{abstract}\hspace{0.1cm}
 In 1996, Bertoin and Werner \cite{BeW96} demonstrated a functional limit theorem, characterising the windings of planar isotropic stable processes around the origin for large times, thereby complementing known results for planar Brownian motion.  The question of windings at small times can be handled using scaling. Nonetheless we examine the case of windings at the the origin using new techniques from the theory of self-similar Markov processes. This allows us to understand  upcrossings of (not necessarily symmetric) stable processes over the origin for large and small times in the one-dimensional setting.
\end{abstract}

\begin{keyword}[class=MSC]
\kwd[Primary ]{60J30}
\kwd{60G18}
\kwd[; secondary ]{60J15}
\end{keyword}

\begin{keyword}
\kwd{Stable processes, winding numbers, self-similarity, Lamperti transform, duality, time change, Riesz--Bogdan--\.{Z}ak transform, upcrossings.}
\end{keyword}

\end{frontmatter}

\section{Introduction}
Any planar stochastic process may be written in polar coordinates, say $(r_t\exp ({\rm i}\vartheta_t), t\geq 0)$.
The  angular part $(\vartheta_t, t\geq 0)$, is often referred to as its {\it winding number}, as its value modulo $2\pi$ tells us the number of times the process has wound (and unwound) around the origin.
Windings of 2-dimensional (planar) processes is a classical topic that has attracted the attention of several researchers over the last  decades.
The starting point is the case of planar Brownian motion where the conformal invariance property plays an important role in the analysis of windings.
For a planar Brownian motion $B$ starting from a point different from the origin, its continuous winding process is well-defined for large times. It  was initially proven by Spitzer \cite{Spi58} the following convergence in distribution:
\beq\label{Spi}
\frac{2}{\log t}\vartheta_{t} \overset{{d}}{\underset{t\rightarrow\infty}\longrightarrow} C_{1},
\eeq
where $C_{1}$ is a standard Cauchy variable.
Note that this result for planar Brownian motion can be extended to the finite dimensional distributions but not in the sense of functional weak convergence. 
Other subsequent important  contributions related to Spitzer's classical result can also be found in \cite{ItMK65,Dur84,PiY86,LGY90,BeW94, Vak11,VaY12}.


Aside from its intrinsic interest, the issue of Brownian windings appear in various applications.  For example in considering the rotation of a planar polymer \cite{VYH11}
and other applications in neuroscience (see e.g. \cite{DiG13}).
In turn, this  has motivated further developments in the Brownian setting, see e.g. \cite{Vak11,VaY12}, as well as in the setting of complex-valued Ornstein-Uhlenbeck processes \cite{Vak15}.
Furthermore, in Financial Mathematics, exponential functionals of Brownian motion, which may be related to the windings of planar Brownian motion, are of special interest (see e.g. \cite{Yor01}).

A natural development in the theory of windings of stochastic processes  pertains to the mathematical exploration of  planar stable processes, whenever the winding process is well-defined. More recent work in this direction  has considered  the stable Kolmogorov process. That  is to say,  a 2-dimensional Markov process having as one of its coordinates a one-dimensional stable L\'{e}vy process and the second coordinate as the primitive of the first; see \cite{PrS15}.  However, the classical analogue of Spitzer's original winding result deals with the isotropic planar stable process. This was originally treated by Bertoin and Werner \cite{BeW96}. Their main result is stated below. 
%
%
%
%

\begin{theo}[Planar stable windings at $\infty$]\label{BWth}
Suppose that $(X_t, t\geq 0)$ is an isotropic planar $\alpha$-stable process, with $\alpha\in(0,2)$, that is issued from a point different from the origin. Write its  polar decomposition as $X_t = |X_t|\exp({\rm i}\theta_t)$, $t\geq 0$.
Then, there exists a constant $c\in(0,\infty)$ such that the process $(|r|^{-1/2}\theta_{\exp(rt)}, t\geq 0)$ converges weakly in the Skorokhod topology on $D([0,\infty),\mathbb{R})$ to $(\sqrt{c}B_t, t\geq 0)$ as $r\to\infty$, where $(B_t, t\geq 0)$ is a standard one-dimensional Brownian motion issued from the origin.
\end{theo}

A fundamental aspect of their approach was the representation of such processes, not as L\'evy processes, but as self-similar Markov processes. In particular, their analysis was driven by the so-called Lamperti representation of self-similar Markov processes; cf \cite{ACGZ}.  In the same setting, recent work of Doney and Vakeroudis \cite{DoV12} gives a different approach by invoking the continuity, with respect to the Skorokhod topology,  of the composition function (cf. \cite{Whi80}). They obtain the results of \cite{BeW96} as well as providing asymptotic winding results for small times, in the form of a functional limit theorem, when the stable process is issued from a point different from the origin and as a distributional limit when the stable process is issued from the origin.

What appears to be missing from this ensemble of results is a functional limit theorem in the spirit of Theorem \ref{BWth} at time zero when the stable process is issued from the origin. In order to discuss this further, we need to be a little careful with the notation $\theta: = (\theta_t, t\geq 0)$.  Indeed, whilst $\theta$ is a real-valued stochastic process, the quantity $\exp({\rm i}\theta_t) = X_t/|X_t|$, $t\geq 0$, only defines its value modulo $2\pi$. In fact,  $(\theta_t, t\geq 0)$ no longer makes sense when the process is issued from the origin as, by time $t$, the process has already undergone an infinite number of windings around the origin in both directions.  Instead we need to talk about angular displacements in relative, rather than absolute, terms. To this end, we shall henceforth work with $\theta_{[a,b]}$, $0<a\leq b<\infty$, which is well defined as the rotational displacement of $X$ over the time interval $[a,b]$. Of course in the setting that $X$ is issued from a point other than the origin, we can continue to write $\theta_t = \theta_{(0,t]}$.

Self-similarity informs us that, for all $c>0$,
$(\theta_{[s,t]}, 0<s\leq t)$ is equal in law to $(\theta_{[c^{-\alpha}s, c^{-\alpha}t]}, 0<s\leq t)$.
Setting $t = {\rm e}^r = c^{\alpha}$, for $r>0$, and $s = {\rm e}^{ur}$, for $u\in[0,1]$, we find
$(\theta_{[{\rm e}^{ru}, {\rm e}^r]}, 0\leq u\leq 1)$ is equal in law to $(\theta_{[{\rm e}^{r(u-1)}, 1]}, 0\leq u\leq 1)$.
As a consequence, Theorem \ref{BWth}  tells us that, in the sense of weak convergence with respect to the Skorokhod topology,
\begin{align}
\lim_{r\to\infty} r^{-1/2}(\theta_{({\rm e}^{-rv}, 1]}, 0\leq v\leq 1)&\stackrel{(law)}{=} \lim_{r\to\infty}r^{-1/2}(\theta_{({\rm e}^{ru}, {\rm e}^r]}, 0\leq u\leq 1) \notag\\&= \lim_{r\to\infty}r^{-1/2}(\theta_{[1,{\rm e}^r]} - \theta_{[1,{\rm e}^{ru}]}, 0\leq u\leq 1) \notag\\& = (\sqrt{c}(B_1 - B_u), 0\leq u\leq 1)\notag\\& \stackrel{(law)}{=}
(\sqrt{c}B_u, 0\leq u\leq 1)\label{sk}
\end{align}
With additional work, one can in principle piecewise extend the Skorokhod convergence from the interval $u\in[0,1]$ to $u\geq 0$ and this would result in the following theorem.

\begin{theo}[Planar stable windings at 0]\label{main}
Suppose that $(X_t, t\geq 0)$ is an isotropic planar $\alpha$-stable process, with $\alpha\in(0,2)$, that is issued from the origin. 
The process $(r^{-1/2}\theta_{(\exp(-rt),1]}, t\geq 0)$ converges weakly in the Skorokhod topology on $D([0,\infty),\mathbb{R})$ to $(\sqrt{c}B_t, t\geq 0)$ as $r\to\infty$, where   $(B_t, t\geq 0)$ is a standard one-dimensional Brownian motion issued from the origin and $c$ is the same constant appearing in Theorem \ref{BWth}.
\end{theo}

In this article, we would like to explore a completely new approach to stable windings that appeals to the intuition of Markov duality. In particular we want to understand the behaviour of stable processes as they wind out of the origin, as they wind in towards the origin (when conditioned to approach the origin continuously), as they wind to infinity and the pathwise relationship between the three. Although we start our analysis with planar stable processes, we see this familiar domain as a training ground from which we can  learn how to transplant the technology of Markov duality in two-dimensions into an analogous setting for one-dimension stable processes. Specifically, we would like to understand  upcrossings of the origin in one-dimension.

To see the intimate connection, we note that  a stable process $X$ in $d$-dimensions can always be expressed in the form $X_t = (|X_t|, {\rm Arg}(X_t))\in [0,\infty)\times\mathbb{S}_{d-1}$ , where ${\rm Arg(X_t)} = X_t/|X_t|$, $t\geq 0$, and  $\mathbb{S}_{d-1}: = \{x\in\mathbb{R}^d: |x| =1\}$.  What amounts to a single winding of $X$ around the origin for $d= 2$, or equivalently a single winding of  $({\rm Arg}(X_t), t\geq 0)$ around $\mathbb{S}_1$, corresponds to a sojourn $-1\rightarrow1\rightarrow -1$ in $\mathbb{S}_0$ for  $({\rm Arg}(X_t), t\geq 0)$ when $d=1$. Noting that, for every upcrossing of the origin, there is a subsequent downcrossing, it becomes clear that windings in two-dimensions is extremely closely related to upcrossings in one-dimension.
Winding behaviour into and out of the origin for the two-dimensional stable process is of particular interest in relation to the setting of upcrossings in the one-dimensional case on account of the fact that, for the latter, the origin is no longer polar when $\alpha\in(1,2)$. Moreover, as we shall shortly see, understanding a way of relating windings at $\infty$ to windings at $0$, other than appealing to the distributional scaling exploited in \eqref{sk}, affords us the opportunity to work more directly with the almost sure results that naturally appear in one dimensional upcrossings, rather than functional distributional convergence.

\medskip

Let us be technically more precise about some of the objects referred to in the previous paragraph.  In the planar setting, let  $\mathcal{G}_t : =\sigma(X_u: u\leq t)$, $t\geq 0$, and, for all $t\geq 0$, $A\in\mathcal{G}_t$,
\begin{equation}
\mathbb{P}^\circ_x(A, t<\tau^{\{0\}}): = \lim_{\epsilon\downarrow0}\mathbb{P}_x(A , t< \tau^{\{0\}}| \tau^{(0,\epsilon)}<\infty),
\qquad |x|>0,
\label{circ}
\end{equation}
where $\tau^{(0,\epsilon)}: =\inf\{s>0: |X_s|<\epsilon\}$ and $\tau^{\{0\}}: = \inf\{t>0 : X_t = 0\}$. The process $(X, \mathbb{P}^\circ_x)$, $x\neq 0$, is a self-similar  Markov process with zero as an absorbing state which can reasonably be called the planar stable process conditioned to be continuously absorbed at the origin. See \cite[Theorem 2.1]{KRS15} for related computations. We have the following result describing windings into the origin. 
\begin{theo}[Winding into and out of the origin]\label{intoorigin} Suppose that $(X^\circ_t, t\leq \tau^\circ)$ has the law of an isotropic planar stable process conditioned to continuously absorb at the origin issued from a point which is randomised according to the distribution $\mathbb{P}_0(X_{\ell_a-}\in {\rm d}x)$, $|x|<a$, for some $a>0$, where
\[
\ell_a = \sup\{s\geq 0: |X_s|\leq a\}
\text{ and }\tau^\circ : = \inf\{t> 0 : X^\circ_t = 0\}.
\]
 In polar form, write $X^\circ_t = |X^\circ_t|\exp({\rm i}\theta^\circ_t)$, $t\geq 0$.
Then, $(\theta_{(t,1]}, t\leq 1)$ under $\mathbb{P}_0$ and $(\theta^\circ_{\tau^\circ - t}, t\leq \tau^\circ)$ have the same  asymptotic behaviour as $t\downarrow0$ in the sense that
$
(r^{-1/2}\theta^\circ_{\tau^\circ - \exp(-rt)}, t\geq 0)
$
converges weakly in the Skorokhod topology on $D([0,\infty),\mathbb{R})$ to $(\sqrt{c}B_t, t\geq 0)$ as $r\to\infty$.
\end{theo}

The last part of the above theorem can be seen as a corollary to Theorem \ref{main}. However, we shall prove the aforesaid statement directly (not as a matter of folly, but because we need instruction for part (ii) of the next theorem in the one-dimensional case) and hence, as far as this article is concerned, Theorem \ref{main} is a corollary to Theorem \ref{intoorigin}.

\medskip

Now suppose that $X$ is a one-dimensional stable process with two-sided jumps and with index $\alpha\in(0,2)$.
Let ${U}_{[a,b]}$, $0<a\leq b<\infty$, be the number of upcrossings in the time interval $[a,b]$. That is to say
\[
U_{[a,b]} = \sum_{a\leq  s\leq b}\mathbf{1}_{(X_s>0, \, X_{s-}<0)}.
\]
We write $U_t = U_{(0,t]}$, $t\geq 0$, when it is appropriately defined.  (The reader will again note that when $X_0 = 0$, there are infinite crossings of the origin and hence this would be an example of where the notation $U_t$ does not make sense.)
We are interested in upcrossings both as time tends to zero and  to infinity in the case $\alpha\in (0,1]$ (in which regime the origin is polar) and as time tends to zero and to the first hitting time of the origin when $\alpha\in(1,2)$ (in which regime the origin is visited almost surely).  We prove strong laws of large numbers for the upcrossing count, which are reminiscent of the scaling that appears in planar windings of stable processes and Brownian motion. 

\begin{theo}[Stable upcrossings]\label{main2}Suppose that $X$ is a one-dimensional stable process with two-sided jumps and with index $\alpha\in(0,2)$.
\begin{enumerate}
\item[(i)] If $\alpha\in(0,1]$, then when $X$ is issued from a point other than the origin,
\begin{equation}
\lim_{t\to\infty}\frac{{U}_{t}}{\log t}=
\frac{\sin(\pi\alpha\rho)\sin(\pi\alpha\hat\rho)}{\alpha\pi\sin(\pi\alpha)}.
\label{oo}
\end{equation}
almost surely, with the understanding that the constant on the right-hand side above is equal to infinity when $\alpha = 1$.

\item[(ii)] If $\alpha\in(0,1]$, then when $X$ is issued from the origin,
\begin{equation}
 \lim_{t\to0}\frac{{U}_{[t,1]}}{\log (1/t)} = \frac{\sin(\pi\alpha\rho)\sin(\pi\alpha\hat\rho)}{\alpha\pi\sin(\pi\alpha)}
 \label{o}
\end{equation}
almost surely.

\item[(iii)] If $\alpha\in(1,2)$, then, when $X$ is issued from a point other than the origin,
\[
\lim_{t\to0}\frac{{U}_{\tau^{\{0\}}-t}}{\log (1/t)}=  \frac{\sin(\pi\alpha\rho)\sin(\pi\alpha\hat\rho)}{\alpha\pi|\sin(\pi\alpha)|}
\]
almost surely, where $\tau^{\{0\}}  = \inf\{t> 0: X_t = 0\}$.
\end{enumerate}

\end{theo}

\noindent In the above theorem, when $\alpha\in(1,2)$ and $X$  is issued from the origin, the reader may expect to see a result for ${U}_{[t,1]}$ as $t\to0$. However, the question of counting upcrossings  does not make sense any more. For this parameter regime,  because $X$ is issued from the origin, $\tau^{\{0\}} = 0$   almost surely. Moreover,  over each time horizon $[0,\varepsilon)$, $\varepsilon>0$, $X$ enjoys a countable infinity of excursions from the origin; and within each excursion there are a countable infinity of upcrossings.

Our computations appeal to three different path transformations to  represent the entrance law of the stable process when issued from the origin. The first is the so-called Riesz--Bogdan--\.{Z}ak transform introduced in \cite{BoZ06} which gives the law of the stable process when passed through the spatial Kelvin transform with  an additional time change.  The latter is equivalent to performing the second transformation that we use, which is the Doob $h$-transform of $X$ that corresponds to conditioning the stable process to continuously absorb at the origin; cf \cite{KRS15}.  The third path transformation appeals to Markov duality in the sense of Nagasawa \cite{N}. In particular we use that the stable process emerging from the origin is dual, in the appropriate sense of time reversal, to the aforementioned case of  a stable process conditioned to absorb continuously at the origin.

\medskip

The rest of the paper is organised as following. In the next section, we discuss the three path decompositions that we use as the key novelty in our analysis. In Section \ref{3} we give the proof of Theorem \ref{main} by passing first through the proof of Theorem \ref{intoorigin}. This establishes the line of reasoning that allows us  in Section \ref{1dwindings} to prove Theorem \ref{main2}.


\section{Stable processes}
We consider an isotropic planar stable L\'{e}vy process $X=(X_{t},t\geq0)$ with stability index $\alpha\in(0,2)$ and probabilities $\mathbb{P}_x$, $x\in \mathbb{R}^2$.
For more details on L\'{e}vy and stable processes see e.g. \cite{B,Don07,Kbook}.
Recall that, following Lamperti \cite{Lam72}, in general we say that a Markov process $X$ taking values in $\mathbb{R}^{2}$, with semigroup $P_t$, $t\geq 0$,
is isotropic if its transition satisfies
\beq
P_{t}(\phi(x),\phi(G))=P_{t}(x,G),\qquad x\in \mathbb{R}^2, G\in\mathcal{B}(\mathbb{R}^2),
\eeq
for any $\phi$ in the group of orthogonal transformations on $\mathbb{R}^2$.
If $\langle\cdot,\cdot\rangle$ stands for the Euclidean inner product, then an isotropic planar stable process has characteristic exponent given by the relation
\[
\mathbb{E}_0\left[\exp\left({\rm i}\langle{z},X_{t}\rangle\right)\right]=\exp\left(-t |{z}|^{\alpha}\right),\qquad t\geq0, {z}\in\mathbb{C}.
\]
Recall also that, as L\'evy processes, isotropic planar stable processes are transient, meaning that, almost surely,
\[
\lim_{t\to\infty}|X_t| = \infty.
\]
Moreover, they are  polar in $\mathbb{R}^2$, in the sense that, for all $x\in\mathbb{R}^2$,
\[
\mathbb{P}_0(X_t = x\text{ for some }t> 0) = 0.
\]

Planar stable processes  are also self-similar Markov processes, i.e. for all $c>0$ and $x\neq 0$,
\begin{equation}
(cX_{c^{-\alpha}t}, t\geq 0 ) \text{ under } \mathbb{P} \text{ is equal in law to } (X, \mathbb{P}_{cx}).
\label{stablescaling}
\end{equation}
As such, they may also be represented via a space-time transformation of a Markov additive process. To be more precise, it can be shown (see e.g. \cite{BeW96,Chy06,DoV12} and the references therein) that
\begin{equation}
X_t = \exp\{\xi_{H_t} + {\rm i}\rho_{H_t}\}, \qquad t\geq 0,
\label{Lamperti-type}
\end{equation}
where
\beq
H_t = \inf\{s\geq 0 : \int_0^s {\rm e}^{\alpha\xi_u}{\rm d}u>t\} = \int_0^t |X_s|^{-\alpha}{\rm d}s
\label{formulaH}
\eeq
and  $(\xi, \rho) = ((\xi_t,\rho_t): t\geq 0)$ is such that $\rho$ is a symmetric L\'evy process and $\xi$ is a L\'evy process correlated to  $\rho$. This means that $(\xi,\rho)$ is a strong Markov process with probabilities $\mathbf{P}_{x,y}$, $x,y\in\mathbb{R}^2$ such that $( \xi_{t+s} - \xi_t, \rho_{t+s} - \rho_t )$, $s\geq 0$ is independent of $\sigma((\xi_u, \rho_u): u\leq s)$ and equal in law to $(\xi, \rho)$ under $\mathbf{P}_{0,0}$; see for example \cite{CPR}.
(Note that, for a general planar self-similar Markov process, one would normally have that the pair  $(\xi,\rho)$ is a Markov additive process such that $\rho$ modulates the increments of $\xi$, however, in this special setting, we have the additional property that the pair is a L\'evy process.)
\\

The isotropic property of $X$ implies  that $(|X_t|, t\geq 0)$ is a positive self-similar Markov process (pssMp);  see for example Chapter 13 of \cite{Kbook}. In particular, when one considers $\xi$ as a lone process, without information about $\rho$, then it is a L\'evy process. With an abuse of notation, we denote its probabilities by $\mathbf{P}_x$, $x\in\mathbb{R}$.
The fact that $\lim_{t\to\infty}|X_t| =\infty$ (due to transience) implies that $\lim_{t\to\infty}\xi_t = \infty$ almost surely. In Theorem 7.1 of Caballero et al. \cite{CPP}, the characteristic exponent of $\xi$ is derived. Indeed, for $z\in\mathbb{R}$,
\[
-\log\mathbf{E}_0[{\rm e}^{{\rm i}z\xi_1}] = :\Psi(z) =2^\alpha\frac{\Gamma(\frac{1}{2}(-{\rm i}z +\alpha ))}{\Gamma(-\frac{1}{2}{\rm i}z)}\frac{\Gamma(\frac{1}{2}({\rm i}z +2))}{\Gamma(\frac{1}{2}({\rm i}z +2-\alpha))}.
\]
It is straightforward to see that this exponent can be analytically extended to the Laplace exponent $\psi_\xi(u): = -\Psi(-{\rm i}u)$ for $-2<u<\alpha$, which is convex, having roots at $u=0$ and $u = \alpha -2$ and exploding at $u = -2$ and $\alpha$.

\bigskip

In this article, the technique  we will develop  predominantly concerns the relationship between $(X,\mathbb{P}_0)$ and the singular law of $X$ conditioned to be continuously absorbed at the origin as defined in (\ref{circ}). The latter was constructed in Theorem 16 of \cite{CR}. Note  that $\psi_\xi(\alpha-2) = 0$, which is needed to apply the aforesaid Theorem.

As well as being described through the limiting procedure \eqref{circ},  it is also the case that the law of a stable process conditioned to continuously absorb at the origin can also be captured by a Doob $h$-transform. For all $t\geq 0$, $x\neq 0$, we have
\begin{equation}
\left. \frac{{\rm d}\mathbb{P}^\circ_x}{{\rm d}\mathbb{P}_x}\right|_{\mathcal{G}_t} = \frac{|X_t|^{\alpha-2}}{|x|^{\alpha-2}}\mathbf{1}_{(t<\tau^{\{0\}})}.
\label{conditioned}
\end{equation}
See \cite[Theorem 2.1]{KRS15} for related computations.
This change of measure ensures that $(X, \mathbb{P}^\circ_x)$, $x\in\mathbb{R}^2\backslash \{0\}$ is again an isotropic self-similar Markov process and therefore has a decomposition in the spirit of (\ref{Lamperti-type}); cf. \cite{CPR}. Let us write $X^\circ = (X^\circ_t, t\geq 0)$ to mean a canonical version of $(X, \mathbb{P}^\circ_x)$, $x\in\mathbb{R}^2\backslash \{0\}$. Moreover, we shall write its polar decomposition as
\begin{equation}
X^\circ_t = \exp\{\xi^\circ_{H^\circ_t} + {\rm i}\rho^\circ_{H^\circ_t}\}, \qquad t\leq \tau^\circ,
\label{comparewith}
\end{equation}
where
\[
H^\circ_t = \inf\{s\geq 0 : \int_0^s {\rm e}^{\alpha\xi^\circ_u}{\rm d}u>t\} = \int_0^t |X^\circ_s|^{-\alpha}{\rm d}s
\]
and $\tau^\circ = \inf\{s>0 : X_s^\circ = 0\}$.
Once again, the process $(\xi^\circ, \rho^\circ)$ is a Markov additive process, where $\rho^\circ$ is the underlying modulation to $\xi^\circ$. Isotropy also ensures that the process $|X^\circ| := (|X^\circ_t|, t\geq 0)$ is again a positive self-similar Markov process and $\xi^\circ$, when observed as a lone process, is a L\'evy process. On account of the Doob $h$-transform of the law of $X^\circ$ with respect to $X$, one easily verifies that it constitutes an Esscher transform with respect to $\xi$. Moreover, the characteristic exponent, $\Psi^\circ$ of $\xi^\circ$ satisfies
\begin{equation}
\Psi^\circ(z) = \Psi(z - {\rm i}(\alpha-2))  = 2^\alpha\frac{\Gamma(\frac{1}{2}(-{\rm i}z +2 ))}{\Gamma(-\frac{1}{2}({\rm i}z +\alpha-2))}\frac{\Gamma(\frac{1}{2}({\rm i}z +\alpha))}{\Gamma(\frac{1}{2}({\rm i}z))} = \Psi(-z),
\label{-x}
\end{equation}
for $z\in\mathbb{R}$. That is to say, $\xi^\circ$ is equal in law to $-\xi$. In fact, one can go a little further than this observation as the next result confirms.

\begin{lem}\label{rhoxi}
The pair $(\xi^\circ, \rho^\circ)$ is equal in law to the pair $(-\xi, \rho)$.
\end{lem}

Before turning to the proof of this lemma, we must cite the recent and beautiful  result of Bogdan and \.{Z}ak  \cite{BoZ06} (building on earlier work of Riesz, see the discussion in Section 3 of \cite{BGR}), which is based on the Kelvin transform. At the heart of the aforesaid result is  the conformal mapping $K:\mathbb{R}^2\backslash\{0\} \mapsto\mathbb{R}^2\backslash\{0\}$ which inverts space through the unit circle. Specifically,
$$Kx=\frac{x}{|x|^{2}}, \qquad x\in\mathbb{R}^2\backslash\{0\}.$$

\begin{theo}[Riesz--Bogdan--\.{Z}ak transform]\label{BoZ}Let
\beq
\eta_t=\inf\{s\geq0:\int^{s}_{0}\frac{{\rm d}u}{|X_{u}|^{2\alpha}}>t\}.
\eeq
The process $(KX_{\eta_t},t\geq0)$ under $\mathbb{P}_x$ is equal in law to  $X^{\circ}$ issued from $X^\circ_0=Kx$.
\end{theo}

Our objective here, however, is to establish a different connection between $(X,\mathbb{P}_0)$ to the process $X^\circ$ (see Lemma \ref{reverse} below), the proof of which will use the above result.

\begin{proof}[Proof of Lemma \ref{rhoxi}]
 From the Riesz--Bogdan--\.{Z}ak representation of $X^\circ$, we can say that
$
(X^\circ_t, t\leq \tau^{\{0\}} )
$
with $X^\circ_0 =x\neq 0$, is equal in law to
\begin{equation}
\exp\{-\xi_{H\circ\eta_t} + {\rm i}\rho_{H\circ\eta_t}\}, \qquad t\leq I_\infty,
\label{compare1}
\end{equation}
where $\xi_0 = -\log x$ and $I_\infty = \int_0^\infty {\rm e}^{-\alpha \xi_u}{\rm d}u$. Note that, for $t\leq I_\infty$,
\[
\int_0^{\eta_t} \frac{1}{|X_{s}|^{2\alpha}}{\rm d}s = t \text{ and } \int_0^{H_t}{\rm e}^{\alpha\xi_u}{\rm d}u = t,
\]
which, in turn, tells us that ${\rm d}\eta_t/{\rm d}t = |X_{\eta_t}|^{2\alpha}$ and ${\rm d}H_t/{\rm d}t = {\rm e}^{-\alpha\xi_{H_t}} = |X_t|^{-\alpha}$. It now follows that
\[
\frac{{\rm d}H\circ\eta_t}{{\rm d}t} = \left.\frac{{\rm d}H_s}{{\rm d}s}\right|_{s = \eta_t}\frac{{\rm d}\eta_t}{{\rm d}t} = |X_{\eta_t}|^{\alpha} = {\rm e}^{\alpha \xi_{H\circ\eta_t}},\qquad t\leq I_\infty,
\]
which is to say that
\begin{equation}
H\circ\eta_t = \inf\{s>0 : \int_0^s {\rm e}^{-\alpha \xi_u}{\rm d}u>t\}, \qquad t\leq I_\infty.
\label{compare2}
\end{equation}
Now comparing \eqref{compare1}, \eqref{compare2} and \eqref{-x} with \eqref{comparewith} one deduces that $(\xi^\circ, \rho^\circ)$ is equal in law to the pair $(-\xi,\rho)$.
\end{proof}

As alluded to above, we somehow want to relate the process $(X,\mathbb{P}_0)$ to the process $X^\circ$.
\begin{lem}\label{reverse}
For each $a>0$, recall $\ell_a = \sup\{s\geq 0: |X_s|\leq a\}$.
Conditionally on the event $\{X_{\ell_a -} =x\}$, where $|x|<a$, the process $(X_{(\ell_a-t)-}, t\leq \ell_a)$ under $\mathbb{P}_0$ is equal in law to $X^\circ$ issued from $X^\circ_0 = x$.
\end{lem}

\begin{proof}
We appeal to a line of reasoning that resonates with the proof of Proposition 1 of \cite{CP} and Theorem 2 of \cite{BS}. Like the aforementioned proofs, our proof is fundamentally based on Nagasawa's theory of time reversal for Markov processes; see \cite{N}.
Specifically, Theorem 3.5 of Nagasawa \cite{N} tells us that the time-reversion of $(X, \mathbb{P}_0)$ from its last passage time $\ell_a$   is that of a time-homogenous Markov process and, moreover, its semigroup agrees with that of $X^\circ$. However, this conclusion only holds  subject to certain conditions which must first be checked and we dedicate the remainder of the proof to verifying what is needed.

It turns out that, once we have verified one of the main conditions stipulated amongst those listed in A.3.1-A.3.3 in Nagasawa \cite{N}, the rest are trivial  to verify. To deal with this principal condition, let us introduce some notation. For $x,y\in\mathbb{R}^2$, we shall write $R(x, {\rm d}y)$ for the resolvent of $X$.
It has been known for a long time (see for example p. 543 of \cite{BGR}), that
\begin{equation}
R(x, {\rm d}y) = C(\alpha)|x-y|^{\alpha -2}, \qquad x,y\in\mathbb{R}^2,
\label{R}
\end{equation}
where $C(\alpha)$ is a constant depending on the index of stability $\alpha$ that is of no interest here.
Taking account of the fact that $X$ is issued from the origin, paraphrasing
the principle condition of Nagasawa \cite{N}, we need to check is that, with 
 \[
\varpi({\rm d}x): = \int_{\mathbb{R}^2}\delta_{0}({\rm d}a)R(a, {\rm d}x)= R(0, {\rm d}x) = |x|^{\alpha -2}{\rm d}x, \qquad x\in \mathbb{R}^2\backslash\{0\},
\]
we have
\begin{equation}
p_t(x, {\rm d}y)\varpi({\rm d}x)=p^{\circ}_t(y, {\rm d}x)\varpi({\rm d}y) , \qquad   x,y\in\mathbb{R}^2\backslash\{0\}, t\geq 0.
\label{nagasawacheck}
\end{equation}
Here, $p_t(x,{\rm d}y)$ and $p_t^\circ(y, {\rm d}x)$ represent the transition semigroups of $X$ and $X^\circ$.

We now see that \eqref{nagasawacheck} requires us to check that
\[
p_t(x, {\rm d}y)|x|^{\alpha-2}{\rm d }x = \frac{|x|^{\alpha-2}}{|y|^{\alpha-2}}p_t(y, {\rm d}x) |y|^{\alpha-2}{\rm d}y,\qquad x,y\in\mathbb{R}^2.
\]
Hence, we require that $p_t(x, {\rm d}y){\rm d }x =p_t(y, {\rm d}x){\rm d}y$, $x,y\in\mathbb{R}^2$. However, this is nothing more than the classical duality property for L\'evy process semi-groups (and in particular for isotropic stable process semi-groups).
\end{proof}

\begin{rem}\rm\label{rem}
The consequence of this last lemma is that we can study the windings of $X$ backwards from last exit from the ball of radius $a$ by considering instead the windings of $X^\circ$ as $t\uparrow\tau^\circ$ from a randomised initial position, which we can control by conditioning on the distribution of aforesaid last exit point. However, because of the nature of the scaling in the winding functional limit theorem and that only finite winding can occur over finite time horizons,  knowledge of  backward winding of  $X$ from $\ell_a$ to the origin is sufficient to tell us about backward winding of $X$ from 1 to the origin. Indeed,
\[
\theta_{[t,1]} = \theta_{[t,\ell_a]} + \theta_{(\ell_a,1]}\mathbf{1}_{(\ell_a\leq 1)} - \theta_{(1,\ell_a]}\mathbf{1}_{(\ell_a>1)},
\]
and hence, when scaling by $r^{-1/2}$ such as is proposed in Theorems \ref{main} and \ref{intoorigin}, the difference $|\theta_{[t,1]} - \theta_{[t,\ell_a]} |$ becomes irrelevant.\hfill$\diamond$
\end{rem}



\section{The winding process}\label{3}
Lemma \ref{reverse} and Remark \ref{rem} thereafter tells us that   studying winding backwards to the origin, $\theta_{[t,1]}$ as $t\downarrow0$, is equivalent to studying the forward winding $\theta^\circ_{\tau^\circ-s}$ as $s\downarrow0$, under $\mathbb{P}^\circ_{\nu_a}: = \int_{|x|<a}\nu_a({\rm d} x)\mathbb{P}^\circ_x$, where $\nu_a({\rm d}x): = \mathbb{P}_0(X_{\ell_a -}\in{\rm d} x)$, where $a>0$ is a fixed constant.
We are therefore interested in a functional limit theorem for $(\theta^\circ_{\tau^\circ-s}, s\leq \tau^\circ)$ as $s\downarrow 0$.

From the  representation (\ref{comparewith}) 
we have that, on $\{s<\tau^\circ\}$
\begin{equation}
\theta^\circ_{\tau^\circ-s} = \rho^\circ_{H^\circ_{\tau^\circ - s}},
\label{backwards}
\end{equation}
where
\[
\int_0^{H^\circ_{\tau^\circ - s}} {\rm e}^{\alpha\xi^\circ_u}{\rm d}u = \tau^\circ - s.
\]
For convenience, let us write $\varphi^\circ_s = H^\circ_{\tau^\circ - s}$, providing $s\leq \tau^\circ$. Note in particular that  $\varphi^\circ_{\tau^\circ} = 0$ and that $\varphi^\circ_{0} = \infty$. We also have that
\[
\int_{\varphi^\circ_s}^\infty  {\rm e}^{\alpha\xi^\circ_u} {\rm d}u =\int_0^\infty  {\rm e}^{\alpha\xi^\circ_u} {\rm d}u -\int_0^{H^\circ_{\tau^\circ - s}} {\rm e}^{\alpha\xi^\circ_u} {\rm d}u = \tau^\circ - (\tau^\circ - s) = s.
\]
Differentiating, we see that, on $\{s<\tau^\circ\}$,
\[
\frac{{\rm d}\varphi^\circ_s}{{\rm d}s} = - {\rm e}^{-\alpha\xi^\circ_{\varphi^\circ_s}}  = -|X^\circ_{\tau^\circ-s}|^{-\alpha}
\]
and hence, after integrating, since $\varphi^\circ_{\tau^\circ} = 0$, on $\{t<\tau^\circ\}$,
\[
\varphi^\circ_t = \varphi^\circ_t - \varphi^\circ_{\tau^\circ} = \int_t^{\tau^\circ}|X^\circ_{\tau^\circ-s}|^{-\alpha}{\rm d}s. 
\]
Now define $\tilde{X}^\circ_v = {\rm e}^{v/\alpha}X^\circ_{\tau^\circ-{\rm e}^{-v}}$, ${\rm e}^{-v}<  \tau^\circ$, so that, on $\{t<\tau^\circ\}$,
\beq
\varphi^\circ_t  = \int_{-\log\tau^\circ}^{-\log t} |\tilde{X}^\circ_v|^{-\alpha}{\rm d}v.
\label{varphicirc}
\eeq
Next, we recall from Lemma \ref{reverse} that $(X^\circ_{(\tau^\circ-s)-}, s\leq \tau^\circ)$ under $\mathbb{P}^\circ_{\nu_a}$, where $\nu_a({\rm d}x) = \mathbb{P}_0(X_{\ell_a -} \in{\rm d}x)$, $|x|<a$, agrees with $(X_s, s<\ell_a)$ under $\mathbb{P}_0$. It therefore follows that, under $\mathbb{P}^\circ_{\nu_a}$, $(\tilde{X}^\circ_v, {\rm e}^{-v}< \tau^\circ)$ is equal in law to $(e^{v/\alpha}X_{e^{-v}},  {\rm e}^{-v}<\ell_a)$ under $\mathbb{P}_0$. We note that, under $\mathbb{P}_0$, $\tilde{X}_v = {\rm e}^{v/\alpha}X_{{\rm e}^{-v}}$, $v\in\mathbb{R}$, is a stationary ergodic Markov process (cf. \cite{Breiman}), with distribution at each time equal to that of $X_1$. Similar  reasoning to that found   in Corollary 1 of \cite{BeW96}, which is fundamentally based on the Ergodic Theorem for stationary processes (c.f. Theorem 6.28 of \cite{Breiman}), gives us that, for each $\varepsilon>0$, on $\{\tau^\circ>\varepsilon\}$, we have $\mathbb{P}^\circ_{\nu_a}$-almost surely,
\begin{equation}
\lim_{r\to\infty}\frac{\varphi^\circ_{\exp(-r)}}{r}  = \lim_{r\to\infty} \frac{1}{r}\int_{-\log\tau^\circ}^{r} |\tilde{X}^\circ_v|^{-\alpha}{\rm d}v= \mathbb{E}_0[|X_1|^{-\alpha}] =2^{-\alpha}\frac{\Gamma(1-\alpha/2)}{\Gamma(1+\alpha/2)}.
\label{loggrowth}
\end{equation}

We are now in a position to prove our main theorem, giving an exact result for the windings of the stable process at the origin.

\begin{proof}[Proofs of Theorems \ref{main} and \ref{intoorigin}] Lemma \ref{reverse} and Remark \ref{rem} gives us the first statement in Theorem \ref{intoorigin}. In order to deduce Theorem \ref{main} as a corollary as we prove the second statement in Theorem \ref{intoorigin},
we focus our attention on    the windings of $(\theta^\circ_{\tau^\circ-s}, s\leq \tau^\circ)$ as $s\to0$. For the latter process, we recall from (\ref{backwards}) that $\theta^\circ_{\tau^\circ-s } = \rho^\circ_{\varphi^\circ_s}$ as $s\downarrow0$ and that $\lim_{t\to\infty}{\varphi^\circ_{\exp(-t)}}/{t}  = 2^{-\alpha}{\Gamma(1-\alpha/2)}/{\Gamma(1+\alpha/2)}$ almost surely. We know  from Theorem \ref{BWth} of Bertoin and Werner \cite{BeW96} that
$(r^{-1/2}\rho_{H_{\exp(rt)}}, t\geq 0)$ converges in the Skorokhod topology  to $(\sqrt{c}B_t, t\geq 0)$ as $r\to\infty$. We also know from Corollary 1 of Bertoin and Werner \cite{BeW96} that $\lim_{t\to\infty} H_{\exp{t}}/t = 2^{-\alpha}{\Gamma(1-\alpha/2)}/{\Gamma(1+\alpha/2)}=:\upsilon_\alpha$ almost surely.
Taking account of the conclusion of Lemma \ref{rhoxi} and (\ref{loggrowth}),
and using the continuity of the composition function with respect to the Skorokhod topology (cf. Whitt \cite{Whi80} and \cite{DoV12}),
we have, as a first application of the latter fact, that  $(r^{-1/2}\rho_{\upsilon_\alpha rt}, t\geq 0)$ converges in the Skorokhod topology  to $(\sqrt{c}B_t, t\geq 0)$ as $r\to\infty$. As a subsequent application of the continuity of the composition operation with respect to the Skorokhod topology, we have  that
$
(r^{-1/2}\theta^\circ_{\tau^\circ - \exp(-rt)}, t\geq 0)
$
converges  in the Skorokhod topology  to $(\sqrt{c}B_t, t\geq 0)$ as $r\to\infty$. That is to say, $(r^{-1/2}\theta_{( \exp(-rt),1]}, t\geq 0)$ converges  in the Skorokhod topology  to $(\sqrt{c}B_t, t\geq 0)$ as $r\to\infty$.
\end{proof}

\section{Upcrossings of one-dimensional stable processes}\label{1dwindings}
In this Section we turn our interest to the one-dimensional case.
Hereafter, $X=(X_{t},t\geq0)$ will denote a one-dimensional stable process which has both positive and negative jumps with stability index $\alpha\in(0,2)$ and probabilities $\mathbb{P}_x$, $x\in \mathbb{R}$.   Note in particular that we do not insist that $X$ is symmetric. To be more precise, $X$ is a one-dimensional L\'evy process which respects the scaling property (\ref{stablescaling}). We take the normalisation of $X$ such that its characteristic exponent
satisfies
\[
-\frac{1}{t}\log \mathbb{E}_0[{\rm e}^{{\rm i}zX_t}] = |z|^\alpha ({\rm e}^{{\rm i}\pi\alpha(1/2-{q})}\mathbf{1}_{\{z\geq 0\}} + {\rm e}^{-{\rm i}\pi\alpha(1/2-\hat{q})}\mathbf{1}_{\{z<0\}}), \qquad z\in\mathbb{R}, t\geq0,
\]
where $\hat{q} = 1-{q}$ and ${q} = \mathbb{P}_0(X_t >0)$. Note that  ${q}$ does not depend on $t>0$ thanks to the scaling property (\ref{stablescaling}) of $|X|$. We henceforth assume that $\alpha{q}$ and $\alpha\hat{q}$ belong to $(0,1)$. This is equivalent to ensuring that $X$ has jumps of both signs.

The long term behaviour of $X$ can differ from its two-dimensional counter part depending on the value of $\alpha$. When $\alpha\in(0,1)$, we know that  $\lim_{t\to\infty}|X_t| = \infty$ and $\mathbb{P}_x(\tau^{\{0\}}=\infty)=1$, $x\neq 0$, where $\tau^{\{0\}} = \inf\{s>0 : X_s = 0\}$. When $\alpha = 1$, we have $\limsup_{t\to\infty}|X_t| = \infty$, $\liminf_{t\to\infty}|X_t| = 0$ and  $\mathbb{P}_x(\tau^{\{0\}}=\infty) = 1$, $x\neq 0$. Finally, when $\alpha\in(1,2)$, we have  $\mathbb{P}_x(\tau^{\{0\}}<\infty) = 1$, $x\neq 0$.

On account of the fact that $X$ is a self-similar Markov process, it follows that, when $X_0\neq 0$,
 there exists
a (c\`{a}dl\`{a}g) Markov additive process (MAP), $(\xi,J) = ((\xi_t, J_t), t\geq 0)$,  
taking values in $\mathbb{R}\times \{-1,1\}$  such that, for $X_0\neq0$,
\beq\label{1dpolar}
X_t=\exp(\xi_{\varsigma_t})J_{\varsigma_t}, \qquad t\leq\tau^{\{0\}},
\eeq
where
\begin{equation}
\varsigma_t=\inf\{s\geq0:\int_{0}^{s}{\rm e}^{\alpha \xi_u}{\rm d}u>t\}=\int_{0}^{t}|X_s|^{-\alpha}{\rm d}s.
\label{varsigma}
\end{equation}
Recall that, when $\alpha\in(0,1]$, $\tau^{\{0\}} = \infty$ almost surely so the decomposition holds for all times, otherwise, when $\alpha\in(1,2)$, it only gives a pathwise decomposition up until first hitting of the origin.
The representation in (\ref{1dpolar}) is known as the {\it Lamperti--Kiu} transform and holds for all real valued self-similar Markov processes up to first absorption at the origin.  The Lamperti--Kiu transform can be thought of as the analogue of the polar decomposition (\ref{Lamperti-type}) for planar stable processes. The MAP $(\xi, J)$ is characterised by a matrix exponent which plays a similar role to the characteristic exponent of L\'evy processes. Specifically, if we denote by $\mathbf{P}_{x,i}$, for $x\in\mathbb{R}$ and $i\in\{-1,1\}$, the probabilities of $(\xi, J)$, then
 \[
\mathbf{E}_{0,i}[e^{z\xi_t};J_t=j ] =   ({\rm e}^{{\boldsymbol F}(z) t})_{i,j}, \qquad i,j\in\{-1,1\}, t\geq 0,
  \]
  where\footnote{Here and throughout the paper the matrix entries are arranged by
    \[
      A=\left(\begin{matrix}
          A_{1,1} & A_{1,-1}\\
          A_{-1,1} & A_{-1,-1}
      \end{matrix}\right).
    \]
  }
  \begin{equation}\label{eq:MAP_expo}
    {\boldsymbol F}({z})=\left(\begin{array}{cc}
    -\dfrac{\Gamma(\alpha-{z})\Gamma(1+{z})}{\Gamma(\alpha\hat{q}-{z})\Gamma(1-\alpha\hat{q}+{z})} & \dfrac{\Gamma(\alpha-{z})\Gamma(1+{z})}{\Gamma(\alpha\hat{q})\Gamma(1-\alpha\hat{q})}\\
    &\\
       \dfrac{\Gamma(\alpha-{z})\Gamma(1+{z})}{\Gamma(\alpha{q})\Gamma(1-\alpha{q})}& -\dfrac{\Gamma(\alpha-{z})\Gamma(1+{z})}{\Gamma(\alpha{q}-{z})\Gamma(1-\alpha{q}+{z})}
    \end{array}\right),
  \end{equation}
  for Re$({z}) \in (-1,\alpha)$; see e.g. \cite{CPR, K}.
Note that the above matrix is indexed ${\boldsymbol F}(z)_{1,1}$ in the top left-hand corner. Note, moreover, that  the $Q$-matrix of $J$ is equal to ${\boldsymbol F}(0)$.

Now let $N := (N_t,t\geq0)$ be the counting process of the number of jumps of the process $J$ from -1 to 1 in the time interval $[0,t]$ when $X$ is issued from a point other than the origin. That is to say,
\[
N_t = \sum_{0\leq s\leq t}\mathbf{1}_{(J_{s-} = -1, \, J_s = 1)}, \qquad t\geq 0.
\] We also define ${U}: = ({U}_{t}, 0\leq t\leq \tau^{\{0\}})$ to be the counting process of the number of upcrossings from $(-\infty,0)$ to $(0,\infty)$ up to time $t$. (Note that, under the assumptions we have made on the class of stable processes we consider, $X$ cannot creep upwards, or indeed downwards, and hence upcrossings into the positive half-line will always be by a jump.) The processes $N$ and ${U}$  are related by the time change ${U}_{t} = N_{\varsigma_t}$, $0\leq t\leq \tau^{\{0\}}$.
For every $n\in\mathbb{N}$, we define
\[
T_n=\inf\{t>0:N_t=n\}.
\]
The random time between two consecutive upcrossings in the time-scale of the MAP is distributed as the sum of two independent exponential variables, the holding times of $J$ between the transitions $1 \rightarrow -1$ and $-1\rightarrow 1$, with respective rates ${\boldsymbol F}(0)_{1,-1}$ and ${\boldsymbol F}(0)_{-1,1}$. Classical renewal theory tells us that, for all $x\in\mathbb{R}$ and $i\in\{-1,1\}$, we have $\mathbf{P}_{x,i}$-almost surely,
\begin{align}
\lim_{t\to\infty}\frac{N_t}{t}&=\lim_{n\to\infty}\frac{n}{T_n}\notag\\
&=\frac{{\boldsymbol F}(0)_{1,-1}{\boldsymbol F}(0)_{-1,1}}{{\boldsymbol F}(0)_{1,-1}+{\boldsymbol F}(0)_{-1,1}} \notag\\
& =\frac{\Gamma(\alpha)}{\Gamma(\alpha\hat{q})\Gamma(1-\alpha\hat{q})  +\Gamma(\alpha{q})\Gamma(1-\alpha{q})}\notag\\
&=\frac{\Gamma(\alpha)\sin(\pi\alpha\rho)\sin(\pi\alpha\hat\rho)}{\pi(\sin(\pi\alpha\rho)+\sin(\pi\alpha\hat\rho))}.
\label{slln}
\end{align}

\bigskip

We prove Theorem \ref{main2} by splitting it into two propositions, which correspond to the first and second part of the theorem, respectively. We consider first the case that $\alpha\in(0,1]$, where we know that $\limsup_{t\to\infty}|X_t| =\infty$ almost surely.
First, we can get easily a large time asymptotic result. 
\begin{prop}
Suppose that $X$ is a one-dimensional stable process with two-sided jumps and with index $\alpha\in(0,1]$. Then, when $X$ is issued from a point other than the origin,
\beq
\lim_{t\to\infty}\frac{{U}_{t}}{\log t}= \frac{\sin(\pi\alpha\rho)\sin(\pi\alpha\hat\rho)}{\alpha\pi\sin(\pi\alpha)}
\label{oo}
\eeq
almost surely and, when $X$ is issued from the orgin,
\beq
 \lim_{t\to0}\frac{{U}_{[t,1]}}{\log (1/t)} = \frac{\sin(\pi\alpha\rho)\sin(\pi\alpha\hat\rho)}{\alpha\pi\sin(\pi\alpha)}
  \label{o}
\eeq
almost surely. In both cases, we understand the constant in the limit as  equal to infinity when $\alpha = 1$.
\end{prop}
\begin{proof}
Recalling that $
{U}_{t}=N_{\varsigma_t}
$, $t\geq 0$, on account of (\ref{slln}), it suffices to prove that $\varsigma_t$ grows like $\log t$, at least in an almost sure sense. The method we use is similar to the analysis of the clock $H$ for planar stable processes in \cite{BeW96}. In particular, it is straightforward to see that $\tilde{X}_v = e^{-\alpha/v}X_{e^v}$, $v\in\mathbb{R}$, under $\mathbb{P}_0$ is a stationary ergodic Markov process with stationary distribution equal to that of $X_1$ and hence, from (\ref{varsigma}) we have
\beq
\hspace{1cm}
\lim_{t\to\infty}\frac{\varsigma_{\exp t}}{t}=\lim_{t\to\infty}\frac{\varsigma_{\exp t} - \varsigma_1}{t}=\frac{1}{t}\int_1^{\exp t} |X_u|^{-\alpha}{\rm d}u = \frac{1}{t}\int_0^{t} |\tilde{X}_v|^{-\alpha}{\rm d}v = \mathbb{E}_0[|X_1|^{-\alpha}].
\label{varsigmaslln}
\eeq
We can compute the expectation $ \mathbb{E}_0[|X_1|^{-\alpha}]$ by recalling the following result from Theorem 2.6.3 in Zolotarev \cite{Z}, which states that,
for all $s$ in the strip $-1<\Re(s)<\alpha$, we have
\[
 {\mathbb E} \left[ X^{s} {\bf 1}_{\{X>0\}} \right]=\frac{\sin(\pi \rho s)}{\sin(\pi s)} \frac{\Gamma(1-s/\alpha)}{\Gamma(1-s)}.
 \]
For $\alpha\in(0,1)$ this leads to
\[
 \mathbb{E}[|X_1|^{-\alpha}] = \frac{\sin(\pi  \alpha\rho)  + \sin(\pi \alpha\hat\rho )}{\Gamma(1+\alpha)\sin(\pi \alpha)}.  \]
Note however, this moment explodes when $\alpha = 1$. One may also  verify directly from the Cauchy density that, for the Cauchy process, $\mathbb{E}[|X_1|^{-1}] = \infty$.

The almost sure limit (\ref{oo})  when $\alpha\in(0,1)$ now follows by combining the two strong laws of large numbers in (\ref{slln}) and (\ref{varsigmaslln}). When $\alpha = 1$, we note that, for each $M>0$, we have for all $t$  sufficiently large that ${\varsigma_{\exp t} }/{t}>M$. Using the monotonicity of the counting process $N$ and the strong law of large numbers in \eqref{slln}, it now follows that
\[
\liminf_{t\to\infty} \frac{U_{\exp t}}{t} \geq \liminf_{t\to\infty} M\frac{N_{Mt}}{M t} >M/2\pi.
\]
Since $M$ can be chosen arbitrarily large, the statement of the theorem also follows for Cauchy processes.

Now suppose that $X_0 = 0$ and we consider the upcrossings of  $X$ as $t\to0$. Appealing to a similar method as for the planar case, we will make use of the Riesz--Bogdan--\.Zak transform that was proved in \cite{K}. As $X$ is no longer isotropic (meaning symmetric in the one-dimensional case), it is slightly more complicated to state.

 To this end, define as before,
\[
\eta(t) = \inf\{s>0 : \int_0^s |X_u|^{-2\alpha}{\rm d}u >t\}, \qquad t\geq 0.
\]
Then, for all $x\in\mathbb{R}\backslash\{0\}$, $(-1/{X}_{\eta(t)})_{t\geq 0}$ under $\mathbb{P}_{x}$
is a rssMp  with underlying MAP via the Lamperti--Kiu transform given by
\begin{equation}
\boldsymbol{F}^\circ(z) =
 \left(
  \begin{array}{cc}
    - \dfrac{\Gamma(1-z)\Gamma(\alpha+z)}
      {\Gamma(1-\alpha{q}-z)\Gamma(\alpha{q}+ z)}
    & \dfrac{\Gamma(1-z)\Gamma(\alpha+z)}
      {\Gamma(\alpha{q})\Gamma(1-\alpha{q})}
    \\
    &\\
    \dfrac{\Gamma(1-z)\Gamma(\alpha+ z)}
      {\Gamma(\alpha\hat{q})\Gamma(1-\alpha\hat{q})}
    & - \dfrac{\Gamma(1-z)\Gamma(\alpha+z)}
      {\Gamma(1-\alpha\hat{q}-z)\Gamma(\alpha\hat{q}+z)}
  \end{array}
  \right),
  \label{Fcirc}
\end{equation}
for $\Re(z)\in(-\alpha,1).$
Moreover, for all $x\in\mathbb{R}\backslash\{0\}$, $(-1/{X}_{\eta(t)})_{t\geq 0}$ under $\mathbb{P}_{x}$ is equal in law to $(X, \mathbb{P}_{-1/x}^\circ)$, where
\begin{equation}\hspace{-1cm}
\left.\frac{{\rm d}\mathbb{P}^\circ_x}{{\rm d}\mathbb{P}_x}\right|_{\mathcal{F}_t} = \frac{h(X_t)
}{h(x)
}\mathbf{1}_{(t<\tau^{\{0\}})}
\label{updownCOM}
\end{equation}
with
\[
h(x) =\left(\sin(\pi\alpha\hat{q})\mathbf{1}_{(x\geq 0)} +
\sin(\pi\alpha{q})\mathbf{1}_{(x<0)}\right) |x|^{\alpha -1}
\]
and $\mathcal{F}_t := \sigma(X_s: s\leq t)$, $t\geq 0$.

It was shown in  \cite{KRS15} that, when $\alpha\in(0,1)$,  the change of measure in \eqref{updownCOM}  corresponds to conditioning $X$ to continuously absorb at the origin.
Appealing to Nagasawa's method of duality we can show that the analogue of Lemma \ref{reverse} also holds here. Indeed, the analogue of (\ref{nagasawacheck}) can be easily checked,
recalling, in particular,  that the resolvent $R(x, {\rm d}y)$, $x,y\in\mathbb{R}$ is known to satisfy $R(x, {\rm d}y) = h(y-x){\rm d}y$ up to a multiplicative constant; see e.g. Kyprianou \cite{ALEAKyp}.

 If we write $X^\circ$ as a canonical version  of the real-valued self-similar Markov process under $ \mathbb{P}^\circ_x$, $x\in\mathbb{R}\backslash\{0\}$, it is now the case that understanding ${U}_{[t,1]}$ as $t\to0$ is equivalent to understanding ${U}^\circ_{\tau^\circ-s}$ as $s\to0$, where ${U}^\circ$ is the number of upcrossings of $X^\circ$ and $\tau^\circ = \inf\{s>0: X_s^\circ = 0\}$. (Note that upcrossings of $X$ corresponds to downcrossings of $X^\circ$, however, every upcrossing is followed by a downcrossing and vice versa.) At this point, we note that the MAP that underlies the process $(X, \mathbb{P}^\circ_\cdot)$ has the property that its modulating chain, say $J^\circ$, has the same $Q$-matrix as $J$, albeit the roles of ${q}$ and $\hat{q}$ are interchanged. This has the effect that, if we write $(N^\circ_t, t\geq 0)$ for the process that gives the running count of the number of switches that $J^\circ$ makes from $-1$ to $1$, then it also respects the same strong law of large numbers as (\ref{slln}). Note in particular that, the right hand side of \eqref{slln} is invariant to exchanging the roles of ${q}$ and $\hat{q}$.

The Lamperti--Kiu representation of $X^\circ$ tells us that, if $\varsigma^\circ_t$ is the time change associated to its underlying MAP, then
\[
\varsigma^\circ_t = \int_{0}^{t}|X^\circ_s|^{-\alpha}{\rm d}s.
\]
A computation similar to the one that leads to the equation (\ref{varphicirc}) also tells us that ${U}^\circ_{\tau^\circ -s} = N^\circ_{\sigma^\circ_s}$ and,  for $s\leq \tau^\circ$,
\[
\sigma^\circ_t =  \int_{-\log\tau^\circ}^{-\log t} |\tilde{X}^\circ_v|^{-\alpha}{\rm d}v, \qquad  t\leq \tau^\circ,
\]
where $\tilde{X}^\circ_v = e^{v/\alpha}X^\circ_{\tau^\circ-e^{-v}}$, $e^{-v}<\tau^\circ$.
Continuing along the lines of the proof of (\ref{loggrowth}), we find that, almost surely,
\[
\lim_{s\to\infty}\frac{\sigma^\circ_{e^{-s}}}{s} = \mathbb{E}_0[|X_1|^{-\alpha}].
\]
Combining the strong law of large numbers for $N^\circ$ with the above almost sure limit, we find that (\ref{o}) holds.
\end{proof}

Examining the proof above for the limit as $t\to0$, one also sees  a path to proof for the upcrossings as $t\to\tau^{\{0\}}$ in the case $\alpha\in(1,2)$. Specifically, we note that ${U}_{\tau^{\{0\}} -s} = N_{\sigma_s}$ where,  for $s\leq \tau^{\{0\}}$,
\[
\sigma_t =  \int_{-\log\tau^{\{0\}}}^{-\log t} |\tilde{X}_v|^{-\alpha}{\rm d}v, \qquad  t\leq \tau^{\{0\}},
\]
where $\tilde{X}_v = e^{v/\alpha}X_{\tau^{\{0\}}-e^{-v}}$, ${\rm e}^{-v}<\tau^{\{0\}}$. In \cite{CPR} it was shown that that when $\alpha\in(1,2)$,  the change of measure in \eqref{updownCOM} corresponds to conditioning $X$ to avoid the origin. Moreover, it was proved in Dereich et al. \cite{DDK}, that, if $X^\circ$ is the canonical process with probabilities $\mathbb{P}^\circ_x$, $x\in\mathbb{R}\backslash\{0\}$, then there is a unique definition of  $(X^\circ, \mathbb{P}^\circ_0)$ in such a way that $X^\circ$ leaves the origin continuously and $\mathbb{P}^\circ_0 = \lim_{|x|\to0}\mathbb{P}^\circ_x$ in the sense of weak convergence on the Skorokhod space.

Continuing again along the lines of the proof of (\ref{loggrowth}), we can use Nagasawa duality to show the following lemma.
\begin{lem}\label{circreverse}
 The time reversed process  $(X_{(\tau^{\{0\}} -s)-},s\leq \tau^{\{0\}} )$ when issued from a randomised point with law $\nu^\circ_a({\rm d}x) : = \mathbb{P}^\circ_0(X^\circ_{\ell_a^\circ-}\in{\rm d}x)$, where $\ell^\circ_a = \sup\{t>0: |X^\circ_t|<a\}$ and $a>0$,  is equal in law to $(X^\circ_t, t< \ell^\circ_a)$ under $\mathbb{P}^\circ_0$.
 \end{lem}
Combining with the preceding reasoning, for each fixed $a>0$,  we will  have the $\mathbb{P}_{\nu^\circ_a}$-almost sure limit
\[
\lim_{s\to\infty}\frac{\sigma_{e^{-s}}}{s} = \mathbb{E}^\circ_0[|X^\circ_1|^{-\alpha}],
\]
providing that the expectation on the right-hand side makes sense.
  In order to compute the expectation $\mathbb{E}^\circ_0[|X^\circ_1|^{-\alpha}]$, thereby showing that it is finite,
we can appeal to reasoning in \cite{Rivero} and \cite{PPR}; see also Section 4.4. of \cite{KKPW}.
Specifically, following the computations in all of these papers (for which some simple facts in  Lemma 12 of \cite{DDK} concerning the construction of $\mathbb{P}^\circ_0$ will be useful) we easily conclude that,
\begin{equation}
\mathbb{E}^\circ_0[f(X^\circ_1)] = \Gamma(-\alpha)\frac{\sin(\pi\alpha\rho)}{\pi}\hat{\mathbf{E}}^\circ_{0,1}\left[f(I^{-1/\alpha})I^{-1}\right]\notag \\
+\Gamma(-\alpha)\frac{\sin(\pi\alpha\hat\rho)}{\pi}\hat{\mathbf{E}}^\circ_{0,-1}\left[f(-I^{-1/\alpha})I^{-1}\right],
\label{omean}
\end{equation}
where $\hat{\mathbf{E}}^\circ_{0,i}$ is expectation with respect to $\hat{\mathbf{P}}^\circ_{0,i}$, which is the law of the dual of the MAP underlying $(X^\circ,\mathbb{P}^\circ_x)$, $x\neq 0$, when issued from $(0,i)$, $i\in\{-1,1\}$ and  $I = \int_0^\infty \exp\{\alpha\xi_u\}{\rm d}u$. It is important to note here that, as  $\mathbb{P}^\circ_0$ corresponds to the law of $X$ conditioned to avoid the origin, the MAP $(\xi, J)$ has the property that $\xi_t\to\infty$ as $t\to\infty$ almost surely under $\mathbf{P}^\circ_{0,i}$. Therefore, under $\hat{\mathbf{P}}^\circ_{0,i}$ we have that $\xi_t\to-\infty$ almost surely, and this is sufficient to conclude that $I<\infty$ almost surely.  In fact, one can verify directly from the matrices \eqref{eq:MAP_expo} and \eqref{Fcirc} that $\hat{\mathbf{P}}^\circ_{0,i}$ is nothing more than the MAP corresponding to the negative of the stable process conditioned to continuously absorb at the origin (which is not surprising given the statement of Riesz--Bogdan--\.Zak transformation). For the special case that $f(x) = |x|^{-\alpha}$, it now follows rather easily from (\ref{omean}) that
\[
\mathbb{E}^\circ_0[|X^\circ_1|^{ -\alpha }]  = \Gamma(-\alpha)\frac{(\sin(\pi\alpha\rho)+\sin(\pi\alpha\hat\rho)) }{\pi}.
\]

Combining with the strong law of large numbers in \eqref{slln}, which remains valid in the current setting, i.e. that $\alpha\in(1,2)$, we have $\mathbb{P}_{\nu^\circ_a}$-almost surely.
\begin{align}
\lim_{t\to\infty} \frac{N_{\varsigma_{{\rm e}^{-t}}}}{t}& = \frac{\Gamma(\alpha)\sin(\pi\alpha\rho)\sin(\pi\alpha\hat\rho)}{\pi(\sin(\pi\alpha\rho)+\sin(\pi\alpha\hat\rho))}\times \Gamma(-\alpha)\frac{(\sin(\pi\alpha\rho)+\sin(\pi\alpha\hat\rho)) }{\pi}\notag\\
& =\frac{\sin(\pi\alpha\rho)\sin(\pi\alpha\hat\rho)}{-\alpha\pi\sin(\pi\alpha)}
\label{a.e.}
\end{align}
where we have used the reflection formula for the gamma function.

Unlike before, we now have the problem that, because of the direction of time-reversal, we cannot use the same trick as in Remark \ref{rem}. A way around this is to first show that $\nu^\circ_a$ is absolutely continuous with respect to Lebesgue measure. As we can vary the value of $a>0$, this would give us \eqref{a.e.} $\mathbb{P}_x$-almost surely, for almost every $x\in\mathbb{R}\backslash\{0\}$. The missing Lebesgue-null set of starting points can be recovered by a simple trick. Suppose $x\neq 0$ is such a point. We can run the stable process until it first enters the interval $(-x/2,x/2)$, which it will do with probability 1. Noting that the first entrance into this interval is almost surely finite and the law of the first entry point is  absolutely continuous with respect to Lebesgue measure (cf. \cite{Rog}), the Lebesgue a.e. behaviour in \eqref{a.e.} now delivers the desired result.

We are thus left with proving that $\nu^\circ_a$ is absolutely continuous with respect to Lebesgue measure. Invoking a simple scaling argument, similar to those that we have already seen, it suffices to show that $\nu^\circ_1$ is absolutely continuous.

To this end, let us consider $b>1$. Thanks to Lemma \ref{circreverse}, we have that,  under $\mathbb{P}_{\nu^\circ_b}$, the random time
\[
\sup\{t>0: |X_{(\tau^{\{0\}} -t)-}|<1\} =  \inf\{t>0: |X_t|<1\} =:\tau^{(-1,1)}
\]
is equal in law to $\ell_1$  under $\mathbb{P}^\circ_0$
and hence the law of  $X_{\tau^{(-1,1)}}$ is equal in law to $\nu^\circ_1$.

Note that $\mathbb{P}^\circ_0(\lim_{t\to\infty}|X^\circ_t| = \infty) = 1$. This follows on account of the fact that, if $(\xi, J)$ is the MAP underlying $X$ through the Lamperti--Kiu tranform, then $(-\xi,J)$ is the MAP underlying $X^\circ$; see for example the proof of Theorem 3.1 in \cite{K}. As a consequence $\nu^\circ_b$ converges weakly to the Dirac measure $\delta_{\pm\infty}$, where  $\pm\infty: = \{\infty\}\cup\{-\infty\}$ is  seen as the one-point compactification of $\mathbb{R}$. Equivalently, because the limit on the right-hand side below exists, for bounded and measurable $f$ on $(-1,1)$,
\[
\int_{(-1,1)}f(x)\nu^\circ_1(\d x ) = \lim_{b\to\infty}\mathbb{E}_{\nu^\circ_b}[f(X_{\tau^{(-1,1)}})] = \lim_{|x|\to\infty}\mathbb{E}_{x}[f(X_{\tau^{(-1,1)}})].
\]
The limit on the right-hand side above can be calculated thanks to \cite{KPW}. Indeed, by inspecting equation (20) of \cite{KPW}, one may easily take the limit to see that
\[
 \lim_{|x|\to\infty}\mathbb{E}_{x}[f(X_{\tau^{(-1,1)}})]
 = \frac{2^{\alpha-1} \Gamma(2-\alpha)}{\Gamma(1-\alpha\hat\rho)\Gamma(1-\alpha\rho)}
 \int_{-1}^1 f(y)(1+y)^{-\alpha\rho}(1-y)^{-\alpha\hat\rho}{\rm d} y,
\]
and hence, $\nu^\circ_1$ is absolutely continuous with respect to Lebesgue measure as required.

We thus reach the following conclusion which covers the third part of Theorem \ref{main2}.
\begin{prop}
Suppose that $\alpha\in(1,2)$ and $X$ is issued from a point other than the origin, then 
\beq
\lim_{t\to0}\frac{{U}_{\tau^{\{0\}}-t}}{\log (1/t)}= \frac{\sin(\pi\alpha\rho)\sin(\pi\alpha\hat\rho)}{\alpha\pi|\sin(\pi\alpha)|}
\label{ooo}
\eeq
almost surely.
\end{prop}

\section*{Acknowledgements} We would like to thank Victor Rivero for bringing the computations in \cite{Rivero, PPR} to our attention and Tomasz \.{Z}ak for useful comments throughout the paper. 


\end{document}